\documentclass[12pt]{amsart}

\usepackage{fullpage}
\usepackage{amsmath}
\usepackage{amssymb}
\usepackage{amsthm}
\usepackage{amstext}
\usepackage{diagbox}
\usepackage{appendix}
\usepackage{perpage}
\usepackage{color}
\usepackage{tikz-cd}
\usepackage[colorlinks=true,linkcolor=blue, citecolor=blue]{hyperref}

\newtheorem{thm}{Theorem}[section]
\newtheorem{lemm}[thm]{Lemma}
\newtheorem{coro}[thm]{Corollary}

\theoremstyle{definition}
\newtheorem{expl}[thm]{Example}

\newtheorem{remark}[thm]{Remark}

\begin{document}

\title{A note on cusp forms and representations of $\mathrm{SL}_2(\mathbb{F}_p)$}

\author{Zhe Chen}

\address{Department of Mathematics, Shantou University, Shantou, China}

\email{zhechencz@gmail.com}

\begin{abstract}
Cusp forms are certain holomorphic functions defined on the upper half-plane, and the space of cusp forms for the principal congruence subgroup $\Gamma(p)$, $p$ a prime, is acted on by $\mathrm{SL}_2(\mathbb{F}_p)$. Meanwhile, there is a finite field incarnation of the upper half-plane, the Deligne--Lusztig (or Drinfeld) curve, whose cohomology space is also acted on by $\mathrm{SL}_2(\mathbb{F}_p)$. In this note we compute the relation between these two spaces in the weight $2$ case.
\end{abstract}

\maketitle

\section{Introduction}\label{section:Intro}

Given a prime $p$, the cusp forms of weight $k$ for the principal congruence subgroup $\Gamma(p):=\mathrm{Ker}(\mathrm{SL}_2(\mathbb{Z})\rightarrow \mathrm{SL}_2(\mathbb{F}_p))$ form a finite dimensional linear space over $\mathbb{C}$, denoted by $S_k(\Gamma(p))$; these holomorphic functions defined on the upper half-plane are objects of considerable interest in number theory. Here we focus on the case $k=2$. The space $S_2(\Gamma(p))$ is acted on by $\mathrm{SL}_2(\mathbb{F}_p)$ in a natural way. We want to understand this space by viewing $\mathrm{SL}_2(\mathbb{F}_p)$ as a finite reductive group.

\vspace{2mm} There is a finite field analogue of the upper half-plane, $\mathbb{P}^1\backslash\mathbb{P}^1(\mathbb{F}_p)$, which is an algebraic curve over $\overline{\mathbb{F}}_p$. The group $\mathrm{SL}_2(\mathbb{F}_p)$ also acts on this curve and its $\ell$-adic cohomology in a natural way, and Drinfeld found that all the discrete series representations of $\mathrm{SL}_2(\mathbb{F}_p)$ appeared in this cohomology. This is one of the starting points of Deligne--Lusztig theory, a geometric approach to the representations of reductive groups over finite fields.

\vspace{2mm} Let $G=\mathrm{SL}_2(\overline{\mathbb{F}}_p)$ with $F$ being the standard Frobenius endomorphism on $G$ over $\mathbb{F}_p$ (so $G^F=\mathrm{SL}_2(\mathbb{F}_p)$). And let $Z$ be the centre of $G$.

\vspace{2mm}  We recall our basic objects briefly. (The details can be found e.g.\ in \cite{Diamond_et_Shurman_book2005} and \cite{Bonnafe_2011_rep_SL2_book}.) First, $S_2(\Gamma(p))$ can be identified with the space of differential $1$-forms on the modular curve in a natural way, and $G^F$ acts on $S_2(\Gamma(p))$ via this identification by M\"{o}bius transformation. Let $\overline{S_2(\Gamma(p))}$ be the dual space of $S_2(\Gamma(p))$ and denote the character of $S_2(\Gamma(p))+\overline{S_2(\Gamma(p))}$ by $S_{2,p}$. Now fix an anisotropic torus $T_a$ and fix a split torus $T_s$. Note that $T_a\cap T_s=Z$. For an irreducible character $\theta_s\in\widehat{T_s^F}$, we put $R_{T_s}^{\theta}:=\mathrm{Ind}_{B^F}^{G^F}\widetilde{\theta_s}$, where $B$ is an $F$-stable Borel subgroup containing $T_s$, and $\widetilde{\theta_s}$ is the trivial extension of $\theta_s$; this gives roughly half of the irreducible characters of $G^F$, and the other ones come from the geometry of $\mathbb{P}^1\backslash\mathbb{P}^1(\mathbb{F}_p)$: Fix a prime $\ell\neq p$; each $\theta_a\in\widehat{T_a^F}$ corresponds to an $\ell$-adic local system on $\mathbb{P}^1\backslash\mathbb{P}^1(\mathbb{F}_p)$, and the trace of alternating sum of the corresponding cohomology is a virtual character $R_{T_a}^{\theta_a}$. These $R_{T_a}^{\theta_a}$ and $R_{T_s}^{\theta_s}$ are called Deligne--Lusztig characters of $G^F$.

\vspace{2mm} We show that (see Theorem~\ref{main theorem}), the character $S_{2,p}$ of $\mathrm{SL}_2(\mathbb{F}_p)$ is a sum of Deligne--Lusztig characters with interesting coefficients: The coefficients are linear polynomials in $p$ and only depend on the residue of $p$ modulo $12$. Moreover (see Corollary~\ref{corollary} and Corollary~\ref{corollary 2}), these coefficients imply that the single space $S_2(\Gamma(p))$ usually does not admit such a decomposition, and that every non-trivial irreducible character of $\mathrm{PSL}_2(\mathbb{F}_p)$ is a summand of $S_{2,p}$ when $p$ is big enough. Our argument is based on a formula due to Jared Weinstein and a tensor product property of the Steinberg character.

\vspace{2mm}\noindent {\bf Acknowledgement.} We thank an anonymous referee for very helpful suggestions. During this work the author is partially supported by the STU funding NTF17021.

\section{Comparing the spaces}

For convenience we assume that $p\geq 7$. (This avoids the case that $\dim S_2(\Gamma(p))=0$.) 

\begin{thm}\label{main theorem}
Let $S_{2,p}$ be the character of the representation space $S_2(\Gamma(p))+\overline{S_2(\Gamma(p))}$ (here $\overline{S_2(\Gamma(p))}$ is the dual of ${S_2(\Gamma(p))}$). Then
\begin{equation*}
S_{2,p}= \sum_{\theta\in \widehat{T^{F}_s};\ \theta|_Z=1} c_{\theta}R_{T_s}^{\theta}+\sum_{\theta\in \widehat{T^{F}_a};\ \theta|_Z=1} c_{\theta}R_{T_a}^{\theta},
\end{equation*}
where $c_{\theta}=ap+b$ with $\{a,b\} \subseteq \frac{1}{12}\mathbb{Z}$ depending only on the residue of $p$ modulo $12$. 
\end{thm}

We shall start with a Mackey formula for a tensor product with the Steinberg character. (Indeed, the study concerning tensoring with the Steinberg character has a long history; see \cite{Hiss-Zalesski_Weil--Steinberg_repn_2009} and \cite{Hiss-Zalesski_pams_tensorsteinberg_2010} for more recent works.) Our original approach is an explicit computation, which is now replaced by the following general lemma due to an anonymous referee:

\begin{lemm}\label{lemma 1}
Let $\mathbf{G}$ be a connected reductive group over a finite field $\mathbb{F}_q$ with $\mathrm{char}(\mathbb{F}_q)=q_1$, whose derived subgroup $\mathbf{G}_{\mathrm{der}}$ is of type $\mathsf{A}_1$. Let $F$ be the geometric Frobenius on $\mathbf{G}$. If $\mathbf{T},\mathbf{T}'$ are two $F$-stable maximal tori of $\mathbf{G}$ and if $\theta'\in\widehat{\mathbf{T'}^F}$, then
\begin{equation*}
(-1)^{\epsilon(\mathbf{G})+\epsilon(\mathbf{T})}\cdot{^*R_{\mathbf{T}}}(\mathrm{St}\cdot R_{\mathbf{T}'}^{\theta'})=\sum_{ \substack{ w\in \mathbf{T}^F\backslash\mathbf{G}^F/\mathbf{T}'^F \\ \mathbf{T}={^w\mathbf{T}'} } } {^w\theta'}+|\mathbf{Z}^F/{\mathbf{Z}^{\circ}}^F|\cdot (-1)^{\epsilon(\mathbf{G})+\epsilon(\mathbf{T}')}\cdot \sum_{ \substack{ \theta\in \widehat{\mathbf{T}^F} \\ \theta|_{\mathbf{Z}^F}=\theta'|_{\mathbf{Z}^F} } } \theta,
\end{equation*}
where ${^*R_{\mathbf{T}}}$ denotes the Deligne--Lusztig restriction functor, $\mathrm{St}$ denotes the Steinberg character, $\epsilon(-)$ denotes the $\mathbb{F}_q$-rank, and  $\mathbf{Z}$ denotes the centre of $\mathbf{G}$.
\end{lemm}
\begin{proof}
Consider the class function $\tilde{R}\colon \mathbf{G}^F\rightarrow\overline{\mathbb{Q}}_{\ell}$ given by $\tilde{R}(g)=R_{\mathbf{T}'}^{\theta'}(g_{ss})-R_{\mathbf{T}'}^{\theta'}(g)$, where $g_{ss}$ is the semisimple part of $g$. Note that $\mathrm{St}\cdot \tilde{R}=0$, so 
$$(-1)^{\epsilon(\mathbf{G})+\epsilon(\mathbf{T})}\cdot{^*R_{\mathbf{T}}}(\mathrm{St}\cdot R_{\mathbf{T}'}^{\theta'})=(-1)^{\epsilon(\mathbf{G})+\epsilon(\mathbf{T})}\cdot{^*R_{\mathbf{T}}}(\mathrm{St}\cdot (R_{\mathbf{T}'}^{\theta'}+\tilde{R})).$$
By \cite[12.18]{DM1991} and \cite[12.7]{DM1991} we can rewrite the above as
\begin{equation*}
(-1)^{\epsilon(\mathbf{G})+\epsilon(\mathbf{T})}\cdot{^*R_{\mathbf{T}}}(\mathrm{St}\cdot (R_{\mathbf{T}'}^{\theta'}+\tilde{R}))={\mathrm{Res}^{\mathbf{G}^F}_{\mathbf{T}^F}}(R_{\mathbf{T}'}^{\theta'}+\tilde{R})={^*R_{\mathbf{T}}}(R_{\mathbf{T}'}^{\theta'}+\tilde{R}).
\end{equation*}
Note that by the Mackey formula (see \cite[11.13]{DM1991}) we have
$${^*R_{\mathbf{T}}}R_{\mathbf{T}'}^{\theta'}=\sum_{ \substack{ w\in \mathbf{T}^F\backslash\mathbf{G}^F/\mathbf{T}'^F \\ \mathbf{T}={^w\mathbf{T}'} } } {^w\theta'}.$$
So now it remains to deal with ${^*R_{\mathbf{T}}}\tilde{R}$. It suffices to compute $\langle R_{\mathbf{T}}^{\theta}, \tilde{R} \rangle_{\mathbf{G}^F}$ for a general $\theta\in\widehat{\mathbf{T}^F}$. Note that, since $\tilde{R}$ is zero for semisimple elements and since $\mathbf{G}_{\mathrm{der}}$ is of type $\mathsf{A}_1$, we have
\begin{equation}\label{formula:1}
\langle R_{\mathbf{T}}^{\theta}, \tilde{R} \rangle_{\mathbf{G}^F}=\frac{1}{|\mathbf{G}^F|}\sum_{ \substack{ 1\neq u\in \mathcal{U}^F} }\sum_{ z\in \mathbf{Z}^F} R_{\mathbf{T}}^{\theta} (zu) \tilde{R}(z^{-1}u^{-1}),
\end{equation}
where $\mathcal{U}$ denotes the set of unipotent elements of $\mathbf{G}$. By the Green function character formula (see \cite[4.2]{DL1976}) we have $R_{\mathbf{T}}^{\theta} (zu)=\theta(z)Q_{\mathbf{T}}(u)$, where $Q_{\mathbf{T}}$ denotes the Green function (the restriction of $R_{\mathbf{T}}^{1}$ to $\mathcal{U}^F$). Meanwhile, as $\mathbf{G}_{\mathrm{der}}$ is of type $\mathsf{A}_1$, we have $|\mathcal{U}^F|=q^2=|\mathbf{G}_{\mathrm{der}}^F|_{q_1'}+1$, and every $1\neq u\in\mathcal{U}^F$ is regular (so $Q_{\mathbf{T}}(u)=1$ by \cite[9.16]{DL1976}). Thus
$$\eqref{formula:1}=\frac{1}{|\mathbf{G}^F|}\sum_{  1\neq u\in \mathcal{U}^F   }\sum_{ z\in \mathbf{Z}^F} \theta(z)\theta'(z^{-1}) \cdot (Q_{\mathbf{T}'}(1)-1)=\frac{|\mathbf{G}_{\mathrm{der}}^F|_{q_1'}}{|\mathbf{G}^F|}(Q_{\mathbf{T}'}(1)-1)\sum_{ z\in \mathbf{Z}^F} \theta(z)\theta'(z^{-1}).$$
By the dimension formula (see \cite[12.9]{DM1991}) $Q_{\mathbf{T}'}(1)=(-1)^{\epsilon(\mathbf{G})+\epsilon(\mathbf{T})}{|\mathbf{G}^F|_{q_1'}}/{|{\mathbf{T}'}^F|}$, and then by checking each of the two cases that $\mathbf{T}'$ splits or not, we see this value always equals to $(-1)^{\epsilon(\mathbf{G})+\epsilon(\mathbf{T})}{|\mathbf{G}_{\mathrm{der}}^F|_{q_1}}+1$. Together with the fact that $|\mathbf{G}^F|=|{\mathbf{Z}^{\circ}}^F|\cdot|\mathbf{G}_{\mathrm{der}}^F|$, this implies  $\frac{|\mathbf{G}_{\mathrm{der}}^F|_{q_1'}}{|\mathbf{G}^F|}(Q_{\mathbf{T}'}(1)-1)=(-1)^{\epsilon(\mathbf{G})+\epsilon(\mathbf{T}')}/|{\mathbf{Z}^{\circ}}^F|$. Therefore
$$\langle R_{\mathbf{T}}^{\theta}, \tilde{R} \rangle_{\mathbf{G}^F}=|\mathbf{Z}^F/{\mathbf{Z}^{\circ}}^F| (-1)^{\epsilon(\mathbf{G})+\epsilon(\mathbf{T}')} \langle \theta|_{\mathbf{Z}^F}, \theta'|_{\mathbf{Z}^F}\rangle_{\mathbf{Z}^F},$$
as desired.
\end{proof}

From now on we let $\mathbf{G}$ be $G=\mathrm{SL}_2$ over $\mathbb{F}_p$. (So $\mathbf{Z}=Z$.) Let $T$ be an $F$-stable maximal torus of $G$ and $\theta\in\widehat{T^F}$; we always assume $T$ is $T_a$ or $T_s$ (a fixed anisotropic torus and a fixed split torus). We denote the unique character of $T^F$ of order $2$ by $\alpha$. Note that $\alpha$ is the ``quadratic residue symbol'', i.e.\ $\alpha(t)=1$ if and only if $t$ is a square in $T^F$; in particular:  If $T=T_s$, then $\alpha|_Z=1$ if and only if $p=1\mod 4$; if $T=T_a$, then $\alpha|_Z=1$ if and only if $p=3\mod 4$.

\begin{remark}\label{lemma 3}
Here we make a summary of the situations when specialising $\mathbf{G}$ to $G=\mathrm{SL}_2$ in Lemma~\ref{lemma 1}. Let $T_1$ be an $F$-stable maximal torus of $G$, and $\theta_1\in\widehat{T_1^F}$ with $\theta_1|_Z=1$. The virtual character $\mathrm{St}\otimes R_{T_1}^{\theta_1}$ is a $\mathbb{Z}$-linear combination of Deligne--Lusztig characters, and the coefficient of $R_{T}^{\theta}$ in the combination for each $\theta$ can be arranged to be: (Note that we do not identify $R_{T}^{\theta}$ with $R_{T}^{\theta^{-1}}$ unless $\theta=\alpha$.)
\begin{itemize}
\item[(a)] The coefficient for $R_{T}^{\theta}$ with $\theta|_Z\neq 1$ is zero;

\item[(b)] let $\theta_s\neq 1\in \widehat{T_s^F}$ be such that $\theta_s|_Z=1$. If $T_1=T_s$, then the coefficient for $R_{T_s}^{\theta_s}$ is $1+\langle \theta_1,\theta_s \rangle_{T_s^F}$;

\item[(c)] let $\theta_s\neq 1\in \widehat{T_s^F}$ be such that $\theta_s|_Z=1$. If $T_1=T_a$, then the coefficient for $R_{T_s}^{\theta_s}$ is $-1$;

\item[(d)] let $\theta_a\neq 1\in \widehat{T_a^F}$ be such that $\theta_a|_Z=1$. If $T_1=T_a$, then the coefficient for $R_{T_a}^{\theta_a}$ is $1-\langle \theta_1,\theta_a \rangle_{T_a^F}$;

\item[(e)] let $\theta_a\neq 1\in \widehat{T_a^F}$ be such that $\theta_a|_Z=1$. If $T_1=T_s$, then the coefficient for $R_{T_a}^{\theta_a}$ is $-1$;

\item[(f)] if $T_1=T_s$, then the coefficient of $R_{T_s}^1$ is $1+\langle \theta_1,1 \rangle_{T_s}$ and the coefficient of $R_{T_a}^1$ is $-1$;

\item[(g)] if $T_1=T_a$, then the coefficient of $R_{T_s}^1$ is $-1$ and the coefficient of $R_{T_a}^1$ is $1-\langle \theta_1,1 \rangle_{T_a}$.
\end{itemize}
\end{remark}

Let $G_{x}\subseteq G^F/Z$ be the subgroup (of order $2$) generated by 
$\begin{bmatrix}
0 & 1\\
-1 & 0\\
\end{bmatrix}Z$, 
$G_{y}\subseteq G^F/Z$ the subgroup (of order $3$) generated by 
$\begin{bmatrix}
0 & 1\\
-1 & -1\\
\end{bmatrix}Z$, 
and $G_{z}\subseteq G^F/Z$ the subgroup (of order $p$) generated by 
$\begin{bmatrix}
1 & 1\\
0 & 1\\
\end{bmatrix}Z$,
then Weinstein's argument in \cite[Page~31]{Weinstein_2007_thesis} implies that $S_{2,p}$ is a sum of the permutation characters:
\begin{equation}\label{formula 2}
S_{2,p}= \overline{\mathbb{Q}}_{\ell}[G^F/Z]-\mathrm{Ind}^{G^F/Z}_{G_{x}}1_{G_{x}}-\mathrm{Ind}^{G^F/Z}_{G_y}1_{G_y}-\mathrm{Ind}^{G^F/Z}_{G_{z}}1_{G_{z}}+2\cdot 1_{G^F/Z}.
\end{equation}
(See also \cite[4.3]{Weinstein_2009_IMRN_HilModForm}, in which the formula is established in the framework of parabolic cohomology).

\vspace{2mm} In order to put the space of cusp forms into the picture of representation theory of a finite reductive group, we need to decompose the above characters of large degree; for this purpose we shall use the following nice formula of tensor product by Steinberg character: (See also \cite[2.1]{Hiss-Zalesski_pams_tensorsteinberg_2010} for a relevant formula with general class functions)

\begin{lemm}\label{lemma 4}
We have $(-1)^{1+\epsilon(T)}\mathrm{St}\otimes R_{T}^{\theta}=\mathrm{Ind}_{T^F}^{G^F}\theta$.
\end{lemm}
\begin{proof}
See \cite[7.3]{DL1976}.
\end{proof}

Now let $\widetilde{G}_*$ be the preimage of $G_*$ along the surjection $G^F\rightarrow G^F/Z$ for each $*\in\{ x, y, z \}$, then \eqref{formula 2} becomes
\begin{equation}\label{formula 3}
S_{2,p}= \mathrm{Ind}^{G^F}_Z1_Z-\mathrm{Ind}^{G^F}_{\widetilde{G}_{x}}1_{\widetilde{G}_{x}}-\mathrm{Ind}^{G^F}_{\widetilde{G}_y}1_{\widetilde{G}_y}-\mathrm{Ind}^{G^F}_{\widetilde{G}_{z}}1_{\widetilde{G}_{z}}+2\cdot 1_{G^F}.
\end{equation}
Here a basic observation is that the generators of $G_{x}$ and $G_{y}$ are semisimple, so we can conjugate $\widetilde{G}_{x}$ and $\widetilde{G}_{y}$ into $T_s^F\cong\mathbb{F}_p^{\times}$ or $T_a^F\cong\mu_{p+1}$, which depends on $p\mod 12$:
\begin{lemm}\label{lemma 5}
We have (up to conjugation in $G^F$):
\begin{itemize}
\item If $p=1 \mod 12$, then both $\widetilde{G}_{x}$ and $\widetilde{G}_{y}$ are in $T_s^F$; \item if $p=5 \mod 12$, then $\widetilde{G}_{x}$ is in $T_s^F$ and $\widetilde{G}_{y}$ is in $T_a^F$; 
\item if $p=7 \mod 12$, then $\widetilde{G}_{x}$ is in $T_a^F$ and $\widetilde{G}_{y}$ is in $T_s^F$; 
\item if $p=11 \mod 12$, then both $\widetilde{G}_{x}$ and $\widetilde{G}_{y}$ are in $T_a^F$.
\end{itemize}
\end{lemm}
\begin{proof}
This follows from direct computations.
\end{proof}

For $*\in\{ x,y\}$, let $T_{*}$ be one of $T_{s}$ and $T_{a}$, and suppose $\widetilde{G}_*$ lies in $T_{*}$. Then \eqref{formula 3} becomes (note that $B^F/\widetilde{G}_{z}=T_s^F/Z$)
\begin{equation}\label{formula 4}
\begin{split}
S_{2,p}=
& \mathrm{Ind}^{G^F}_Z1_Z-\mathrm{Ind}^{G^F}_{\widetilde{G}_{z}}1_{\widetilde{G}_{z}}+2\cdot 1_{G^F}-\sum_{\theta\in \widehat{T_{x}^F};\ \theta|_{\widetilde{G}_{x}}=1 } \mathrm{Ind}^{G^F}_{T^F_{x}}\theta-\sum_{\theta\in \widehat{T_{y}^F};\ \theta|_{\widetilde{G}_{y}}=1 } \mathrm{Ind}^{G^F}_{T^F_{y}}\theta\\
= & \sum_{ \theta\in\widehat{T^F_s};\ \theta|_Z=1 } \mathrm{St}\otimes R_{T_s}^{\theta}-\sum_{ \theta\in\widehat{T^F_s};\ \theta|_Z=1 } R_{T_s}^{\theta}+2\cdot 1_{G^F}\\
&-(-1)^{\epsilon(T_{x})+1}\sum_{\theta\in \widehat{T_{x}^F};\ \theta|_{\widetilde{G}_{x}}=1 } \mathrm{St}\otimes R_{T_{x}}^{\theta}-(-1)^{\epsilon(T_y)+1}\sum_{\theta\in \widehat{T_{y}^F};\ \theta|_{\widetilde{G}_{y}}=1 } \mathrm{St}\otimes R_{T_y}^{\theta},
\end{split}
\end{equation}
where the second equality follows from Lemma~\ref{lemma 4}.

\begin{proof}[Proof of Theorem~\ref{main theorem}]
We can write out these $c_{\theta}$ explicitly. Consider the following (possibly empty) subsets of $ \widehat{T^F_*}$ for each $*\in\{ s,a \}$: First, let $A_*$ be the set consisting of those $\theta$ such that $\theta$ is defined and non-trivial on both $\widetilde{G}_{x}$ and $\widetilde{G}_y$, then let $B_*$ be the set consisting of those $\theta\neq 1$ such that $\theta$ is defined and trivial on both $\widetilde{G}_{x}$ and $\widetilde{G}_y$; let $C_*\subseteq \widehat{T^F_*}\setminus (A_*\cup B_*)$ be the set consisting of those $\theta\neq 1$ such that $\theta$ is defined and trivial on ${\widetilde{G}_{x}}$; let $D_*\subseteq \widehat{T^F_*}\setminus (A_*\cup B_* \cup C_*)$ be the set consisting of those $\theta\neq 1$ such that $\theta$ is defined and trivial on ${\widetilde{G}_{y}}$; let $E_*=\{ 1 \}$.

\vspace{2mm} Now by applying Lemma~\ref{lemma 1} (Remark~\ref{lemma 3}) and Lemma~\ref{lemma 5} to \eqref{formula 4} we can decompose $S_{2,p}$ into a sum of $R_{T_*}^{\theta}$ with $\theta|_{Z^F}=1$, where the coefficients $c_{\theta}$ are given as:

\begin{center}
\begin{tabular}{ |c|c|c|c|c| } 
\hline
\diagbox{$\theta$}{$p$} & $1\mod 12$ & $5 \mod12$ & $7\mod 12$ & $11\mod 12$ \\ 
\hline
$A_s{^{^{^{^{}}}}_{_{_{}}}}$ & $\frac{p-1}{12}+1$ & $\frac{p-5}{12}$ & $\frac{p-7}{12}$ & $\frac{p-11}{12}$ \\ 
\hline
$B_s{^{^{^{^{}}}}_{_{_{}}}}$ & $\frac{p-1}{12}-2$ & $\frac{p-5}{12}$ & $\frac{p-7}{12}$ & $\frac{p-11}{12}$ \\
\hline
$C_s{^{^{^{^{}}}}_{_{_{}}}}$ & $\frac{p-1}{12}-1$ & $\frac{p-5}{12}-1$ & $\frac{p-7}{12}$ & $\frac{p-11}{12}$ \\ 
\hline
$D_s{^{^{^{^{}}}}_{_{_{}}}}$ & $\frac{p-1}{12}-1$ & $\frac{p-5}{12}$ & $\frac{p-7}{12}-1$ & $\frac{p-11}{12}$ \\ 
\hline
$E_s{^{^{^{^{}}}}_{_{_{}}}}$ & $\frac{p-1}{12}-1$ & $\frac{p-5}{12}$ & $\frac{p-7}{12}$ & $\frac{p-11}{12}+1$ \\ 
\hline
$A_a{^{^{^{^{}}}}_{_{_{}}}}$ & $-\frac{p-1}{12}$ & $-\frac{p-5}{12}+1$ & $-\frac{p-7}{12}$ & $-\frac{p-11}{12}$ \\ 
\hline
$B_a{^{^{^{^{}}}}_{_{_{}}}}$ & $-\frac{p-1}{12}$ & $-\frac{p-5}{12}+1$ & $-\frac{p-7}{12}$ & $-\frac{p-11}{12}-2$ \\ 
\hline
$C_a{^{^{^{^{}}}}_{_{_{}}}}$ & $-\frac{p-1}{12}$ & $-\frac{p-5}{12}+1$ & $-\frac{p-7}{12}-1$ & $-\frac{p-11}{12}-1$ \\ 
\hline
$D_a{^{^{^{^{}}}}_{_{_{}}}}$ & $-\frac{p-1}{12}$ & $-\frac{p-5}{12}$ & $-\frac{p-7}{12}$ & $-\frac{p-11}{12}-1$ \\ 
\hline
$E_a{^{^{^{^{}}}}_{_{_{}}}}$ &  $-\frac{p-1}{12}+1$ & $-\frac{p-5}{12}$ & $-\frac{p-7}{12}$ & $-\frac{p-11}{12}-1$ \\ 
\hline
\end{tabular}
\end{center}

So the theorem follows.
\end{proof}

The coefficients $c_{\theta}$ in the above argument also imply that, unlike the sum $S_{2,p}$, the single space $S_2(\Gamma(p))$ is usually not uniform (in the sense of \cite[2.15]{Lusztig_whiteBk}). For instance, we have:

\begin{coro}\label{corollary}
The character $\mathrm{Tr}(-,S_2(\Gamma(p)))$ of $G^F=\mathrm{SL}_2(\mathbb{F}_p)$ is not a linear combination of Deligne--Lusztig characters of $\mathrm{SL}_2(\mathbb{F}_p)$ if $p=23\mod 24$.
\end{coro}
\begin{proof}
From the argument of Theorem~\ref{main theorem} we see that the multiplicity of each irreducible constituent of $R_{T_s}^{\alpha}$ in $S_{2,p}$ is an odd integer. As these constituents are not linear combinations of the $R_{T}^{\theta}$'s and as $R_{T_s}^{\alpha}$ is a real character, the corollary follows.
\end{proof}

\begin{expl}
There is an accidental case: Let $p=7$, then $\mathrm{Tr}(-,S_2(\Gamma(7)))$ is an irreducible constituent of $R_{T_a}^{\alpha}$, hence not uniform. However, note that $\mathrm{PSL}_2(\mathbb{F}_7)\cong \mathrm{GL}_3(\mathbb{F}_2)$, so we can also view $\mathrm{Tr}(-,S_2(\Gamma(7)))$ as a character of $\mathrm{GL}_3(\mathbb{F}_2)$, of which it is a cuspidal Deligne--Lusztig character of degree $3$.
\end{expl}

\begin{coro}\label{corollary 2}
Suppose $p\geq 23$. An irreducible character $\rho$ of $G^F/Z=\mathrm{PSL}_2(\mathbb{F}_p)$ is a summand of $S_{2,p}$ if and only if $\rho\neq 1_{G^F/Z}$.
\end{coro}
\begin{proof}
This follows from the expressions of the $c_{\theta}$ in the argument of Theorem~\ref{main theorem}.
\end{proof}

\section{A further remark}

It would be interesting to know whether there are similar results for the principal congruence subgroups $\Gamma(p^r)$ for all $r\in\mathbb{Z}_{>0}$, in which case the representations of $\mathrm{SL}_2(\mathbb{Z}/p^r)\cong \mathrm{SL}_2(\mathbb{Z}_p/p^r)$ (i.e.\ the smooth representations of $\mathrm{SL}_2(\mathbb{Z}_p)$) are involved. Note that there are generalisations of Deligne--Lusztig theory to this setting; see e.g.\ \cite{Lusztig2004RepsFinRings}, \cite{Sta2009Unramified}, \cite{Sta2011ExtendedDL}, and \cite{Chen_2017_InnerProduct}. Also note that, while we are yet lacking of a good knowledge of values of the generalised Deligne--Lusztig characters, Weinstein's formula \eqref{formula 2} still holds, and there are also possible candidates of the Steinberg representation, like the ones in \cite{Lees1978Steinberg} and \cite{Camp2007Steinberg2}.

\bibliographystyle{alpha}
\bibliography{zchenrefs}

\end{document}